\documentclass[fleqn,10pt]{olplainarticle}
\usepackage{amsthm}
\title{Observations on cycles in a variant of the Collatz Graph}

\author[1]{Q Le}
\author[1]{Edward Smith}
\affil[1]{King's College London School of Mathematics}


\begin{abstract}
It is well known that the Collatz Conjecture can be reinterpreted as the Collatz Graph with root vertex 1, asking whether all positive integers are within the tree generated. It is further known that any cycle in the Collatz Graph can be represented as a tuple, given that inputting them into a function outputs an odd positive integer; yet, it is an open question as to whether there exist any tuples not of the form $(2,2,...,2)$, thus disproving the Collatz Conjecture. In this paper, we explore a variant of the Collatz Graph, which allows the 3x+1 operation to be applied to both even and odd integers. We prove an analogous function for this variant, called the Loosened Collatz Function (LCF), and observe various properties of the LCF in relation to tuples and outputs. We prove a certain underlying unique factorisation monoid structure for tuples to the LCF and provide a geometric interpretation of satisfying tuples in higher dimensions. Research into this variant of the Collatz Graph may provide reason as to why there exist no cycles in the Collatz Graph.
\end{abstract}

\begin{document}

\flushbottom
\maketitle
\thispagestyle{empty}

\section{Introduction}
The Collatz function is defined as

$c(n) = \begin{Bmatrix}
\frac{n}{2} &\text{if}& n \equiv 0 \text{ (mod 2)}
\\
3n+1 &\text{if}& n \equiv 1 \text{ (mod 2)}
\end{Bmatrix}$\\
for $n  \in \mathbb{N}_1$. Let $(x, c(x), c^2(x),...)$ be the trajectory of $x$ under the Collatz function. The Collatz Conjecture asks whether, for all $x \in \mathbb{N}_1$, the trajectory of $x$ under the Collatz function yields $1$. Recent empirical evidence (\citeauthor{verification}) has tested the first $2^{68}$ starting values and no counterexample has been found as of yet. It is well known that this conjecture can be reformulated in graph-theoretic terms as the Collatz graph, a graph that is defined by the inverse relation

$C(n) = \begin{Bmatrix}
\{2n\} &\text{if}& n \equiv 0,1,2,3,5 \text{ (mod 6)}
\\
\{2n,\frac{n-1}{3}\} &\text{if}& n \equiv 4 \text{ (mod 6)}
\end{Bmatrix}$\\
for $n \in \mathbb{N}_1$, where the vertices of the graph are positive integers. Thus, the Collatz Conjecture is reframed to the question of whether the Collatz graph is a tree containing all positive integers with root vertex 1.

We define $f,g: \mathbb{N}_1 \to \mathbb{N}_1$ to be the two mappings in the Collatz function that connect one vertex to another, $f:n\to \frac{n}{2}$ and $g:n \to 3n+1$. We define cycles to be simple directed circuits - we will specify when we refer to the more general directed circuit. We also say that $f$ or $g$ is applied in a cycle if $f$ or $g$ is used on one vertex to connect with another. If the Collatz graph contains cycles aside from the trivial $\dots\xrightarrow{f}1 \xrightarrow{g} 4 \xrightarrow{f} 2 \xrightarrow{f}\dots$, then the Collatz Conjecture is false, since there then exists a trajectory that never yields 1.

It was proven (\citeauthor{cycleproofs}) that, for a cycle to exist in the Collatz graph, there must exist a unique tuple $Y = (y_1,\dots,y_n)$, where each entry represents the number of times you apply $f$ after applying $g$ in the cycle. Furthermore, for such a tuple $Y$, it represents a cycle in the Collatz graph if and only if the following function, which we call the Collatz Function, outputs an odd number that has $g$ applied to it in the cycle represented:

$F(Y) = \frac{\sum_{i=1}^{n}3^{n-i}\cdot 2^{\sum_{j=0}^{i-1}y_j}}{2^{\sum_{i=1}^{n}y_i}-3^n}$\\
where $y_0 = 0$. It is important to note that:
\begin{enumerate}
    \item each entry of $Y$ must be greater than 0, since having 0 means you applied $g$ twice consecutively, once on an odd integer, once on an even integer, which is not permitted by the Collatz function.
    \item $y_0=0$ is required for the function to work and should not be seen as an entry for the tuple, since $y_0$ always equals as $0$ regardless of the inputted tuple.
\end{enumerate}
We say that a tuple is an $n$-tuple if it has $n$  entries and a cycle is an $n$-cycle if, within the cycle, $g$ is applied $n$ times. For example, the trivial cycle $\dots\xrightarrow{f}1 \xrightarrow{g} 4 \xrightarrow{f} 2 \xrightarrow{f}\dots$ can be represented by the 1-tuple (2) since the output of the function is

$F((2)) = \frac{1}{2^{2} - 3} = 1$. \\
and $g$ is applied one time in that cycle. In fact, the $n$-tuple (2,2,...,2) represents the trivial cycle, repeated $n$ times. This can be seen by investigating what happens when all entries of the tuple $(a,\dots,a)$ are equal.

Many variants of the Collatz function has been researched, from extending the domain to include integer, rational, real, and complex inputs to generalising the mappings in the Collatz function. In this paper, we investigate a modified Collatz graph called the Loosened Collatz graph, which is defined by the relation

$L(n) = \begin{Bmatrix}
\{2n\} &\text{if}& n \equiv 0,2 \space \text{ (mod 3)}
\\
\{2n,\frac{n-1}{3}\} &\text{if}& n \equiv 1 \space \text{ (mod 3)}
\end{Bmatrix}$ \\
for $n \in \mathbb{N}_1$, with the exception that $L(1)=2$. This is equivalent to modifying the Collatz function into the graph generating relation

$l(n) = \begin{Bmatrix}
\{\frac{n}{2},3n+1\} &\text{if}& n \equiv 0 \space \text{ (mod 2)}
\\
\{3n+1\} &\text{if}& n \equiv 1 \space \text{ (mod 2)}
\end{Bmatrix}$ \\
for $n \in \mathbb{N}_1$. This can be observed by noting that, for all $x \equiv 1 \text{ (mod 3)}$, $\frac{x-1}{3}$ is either even or odd and for all $y \in \mathbb{N}_1$, one can easily construct $3y+1 \equiv 1 \text{ (mod 3)}$. 

From this, we can see that the Collatz Function changes too, by extending the domain of the entries of the tuple to include 0. Thus, for a cycle to exist in the Loosened Collatz Graph, there must exist a unique tuple $Y = (y_1,\dots,y_n)$, where each entry of the tuple represents the number of times you apply $f$ after applying $g$ in the cycle, which may now be 0. Further, for such a tuple $Y$, it represents a cycle in the Collatz graph if and only if the following function, which we call the Loosened Collatz Function (LCF), outputs an integer, even or odd, that has $g$ applied to it in the cycle represented:

$F_L(Y) = \frac{\sum_{i=1}^{n}3^{n-i}\cdot 2^{\sum_{j=0}^{i-1}y_j}}{2^{\sum_{i=1}^{n}y_i}-3^n}$\\
where $y_0 = 0$. From this, we say that a tuple satisfies the LCF if:
\begin{enumerate}
    \item Each entry of the tuple is a non-negative integer.
    \item Inputting the tuple into the LCF outputs a positive integer.
\end{enumerate}
Unlike the original Collatz graph, the Loosened Collatz graph has many cycles that satisfy the LCF. A trivial case is the circuit

$\dots\xrightarrow{f}4 \xrightarrow{g} 13 \xrightarrow{g} 40 \xrightarrow{f} 20 \xrightarrow{f} 10 \xrightarrow{f} 5 \xrightarrow{g} 16 \xrightarrow{f} 8 \xrightarrow{f} 4 \dots$\\ 
which can be represented by the tuple $(0,3,2)$,$(2,0,3)$ and $(3,2,0)$. 

This paper investigates the Loosened Collatz Graph and makes several observations and conjectures regarding it. Initially, we present a derivation of the LCF from the Loosened Collatz Graph relation. Then, we prove the satisfaction of the LCF of bitwise rotations of a tuple, given that a tuple satisfies the LCF, and vice versa. We then look at the underlying monoid structure of satisfying tuples and interpret the tuples as coordinates in higher dimensional spaces, with the hopes that they are useful for other researchers. We conclude with a summary of our research and further research questions to explore.
\newpage
\section{Cycles in the Loosened Collatz Graph}

The Loosened Collatz graph is the graph defined by the relation

$L(n) = \begin{Bmatrix}
\{2n\} &\text{if}& n \equiv 0,2 \space \text{ (mod 3)}
\\
\{2n,\frac{n-1}{3}\} &\text{if}& n \equiv 1 \space \text{ (mod 3)}
\end{Bmatrix}$ \\
for $n \in \mathbb{N}_1$, with the exception that $L(1) = 2$ to avoid having $0$ as a vertex of the graph. We define $f,g$ as in the introduction. By inversing the relation, we have the graph generating relation

$l(n) = \begin{Bmatrix}
\{\frac{n}{2},3n+1\} &\text{if}& n \equiv 0 \space \text{(mod 2)}
\\
\{3n+1\} &\text{if}& n \equiv 1 \space \text{(mod 2)}
\end{Bmatrix}$\\
for $n \in \mathbb{N}_1$. From this, we can obtain the Loosened Collatz Function (LCF).

\subsection{Proof of The Loosened Collatz Function}

\newtheorem{LCF}{Theorem}
\begin{LCF}
For a cycle to exist in the Loosened Collatz graph, its tuple must satisfy the function

$F_L(Y) = \frac{\sum_{i=1}^{n}3^{n-i}\cdot 2^{\sum_{j=0}^{i-1}y_j}}{2^{\sum_{i=1}^{n}y_i}-3^n}$ \\
where $y_0 = 0$, each entry of the tuple $(y_1,\dots,y_n)$ corresponds to the number of times you apply $f$ after applying $g$ in the cycle, including 0, and the output is one of the positive integers in that cycle, which has $g$ applied to it.
\end{LCF}

Although an analogous theorem has been proven true for the original Collatz graph, it will be useful to prove it in this context, since some constructions within the proof will be used later in the paper.

\begin{proof}
Assume there exists an $n$-cycle, in which $g$ is applied $n$-times. Let $x_1$ denote one of the numbers in the cycle to which $g$ is applied. After $g$ is applied onto $x_1$, there exist $y_1$ applications of $f$ before applying $g$ again. Let $x_2 =f^{y_1}(g(x_1))$ be the number before applying $g$ again. After $g$ is applied onto $x_2$, there exist $y_2$ applications of $f$ before applying $g$ again. Let $x_3 = f^{y_2}(g(x_2))$ be the number before applying $g$ again. We may repeat this process until we obtain $x_1 = f^{y_n}(g(x_n))$. It is clear to see that $y_1,\dots,y_n$ are positive integers. So, we obtain the family of equations

$x_2 = f^{y_1}(g(x_1))$

$x_3 = f^{y_2}(g(x_2))$

$\vdots$

$x_n = f^{y_{n-1}}(g(x_{n-1})$

$x_1 = f^{y_n}(g(x_n))$ \\
Notice that,

$x_2 = \frac{3x_{1}+1}{2^{y_1}}$

$x_3 = \frac{3^2x_{1}+2^{y_1}+3}{2^{y_1+y_2}}$

$x_4 = \frac{3^3x_{1}+3\cdot2^{y_1}+2^{y_1+y_2}+3^2}{2^{y_1+y_2+y_3}}$\\
To generalise the values in the numerator and denominator, we shall prove the following lemma:

\newtheorem{MiniCycleLemma}{Lemma}
\begin{MiniCycleLemma}
for $1\leq k \leq n$, $x_k = \frac{3^{k-1}x_{1}+\sum_{i=1}^{k-1}3^{k-i-1}\cdot 2^{\sum_{j=0}^{i-1}y_j}}{2^{\sum_{i=1}^{k-1}y_i}}$, where $y_0=0$.
\end{MiniCycleLemma}

\begin{proof}
We prove this statement by induction. For base cases 2 and 3, the claim is true. Assume that the claim is true for $k=z$, for $1\leq z<n$. So,

$x_z = \frac{3^{z-1}x_{1}+\sum_{i=1}^{z-1}3^{z-i-1}\cdot 2^{\sum_{j=0}^{i-1}y_j}}{2^{\sum_{i=1}^{z-1}y_i}}$.\\
We may consider the case for $k=z+1$.  Since $x_{z+1}=f^{z}(g(x_{z}))$, by performing $g$ and $f$, we obtain

$x_{z+1}= \frac{3(\frac{3^{z-1}x_{1}+\sum_{i=1}^{z-1}3^{z-i-1}\cdot 2^{\sum_{j=0}^{i-1}y_j}}{2^{\sum_{i=1}^{z-1}y_i}})+1}{2^{y_z}}$.\\
Simplifying the fraction, we get

$x_{z+1} = \frac{3^{z}x_{1}+\sum_{i=1}^{z-1}3^{z-i}\cdot 2^{\sum_{j=0}^{i-1}y_j}+2^{\sum_{i=1}^{z-1}y_i}}{2^{\sum_{i=1}^{z-1}y_i+y_z}}$.\\
We may allow $3^{z-z}$ to multiply $2^{\sum_{i=1}^{z-1}y_i}$, since this does not change the overall result.

$x_{z+1} = \frac{3^{z}x_{1}+\sum_{i=1}^{z-1}3^{z-i}\cdot 2^{\sum_{j=0}^{i-1}y_j}+3^{z-z}\cdot2^{\sum_{i=1}^{z-1}y_i}}{2^{\sum_{i=1}^{z-1}y_i+y_z}}$.\\
Observe that $3^{z-z}\cdot2^{\sum_{i=1}^{z-1}y_i}$ would be the term generated by the sum $\sum_{i=1}^{z}3^{z-i}\cdot 2^{\sum_{j=0}^{i-1}y_j}$ for $i=z$. From this, we may write

$x_{z+1} = \frac{3^{z}x_{1}+\sum_{i=1}^{z}3^{z-i}\cdot 2^{\sum_{j=0}^{i-1}y_j}}{2^{\sum_{i=1}^{z}y_i}}$,\\ which makes the claim true for $k=z+1$. By the principle of induction, we have proven the claim true for all $k\in\mathbb{N}_1$.
\end{proof}
\noindent So, by Lemma 1, we have proven that

$x_n = \frac{3^{n-1}x_{1}+\sum_{i=1}^{n-1}3^{n-i-1}\cdot 2^{\sum_{j=0}^{i-1}y_j}}{2^{\sum_{i=1}^{n-1}y_i}}$, \\
and, if $x_1=f^{y_n}(g(x_n))$, then
$x_{1} = \frac{3^{n}x_{1}+\sum_{i=1}^{n}3^{n-i}\cdot 2^{\sum_{j=0}^{i-1}y_j}}{2^{\sum_{i=1}^{n}y_i}}$. Rearranging for $x_1$, we get

$x_{1} = \frac{\sum_{i=1}^{n}3^{n-i}\cdot 2^{\sum_{j=0}^{i-1}y_j}}{2^{\sum_{i=1}^{n}y_i}-3^{n}}$.\\
Since this is true for any number in a cycle on which $g$ is applied, we may allow the equation to become the function
$F_L(Y) = \frac{\sum_{i=1}^{n}3^{n-i}\cdot 2^{\sum_{j=0}^{i-1}y_j}}{2^{\sum_{i=1}^{n}y_i}-3^{n}}$, 
where $y_0 = 0$ regardless of the tuple $Y$, each entry of the tuple $(y_1,\dots,y_n)$ corresponds to the number of times you apply $f$ after applying $g$, including 0, and the output is one of the positive integers in that cycle, which has $g$ applied to it.
\end{proof}

Theorem 1 shows that we have extended the original Collatz Function to allow 0s as entries of the tuple when considering the Loosened Collatz graph.

A natural question to ask is whether there exists an algorithm to find tuples that satisfy the LCF. We provide a partial answer to this - there do exist two methods for finding cycles in the graph, which lend themselves to finding satisfying tuples. However, there is no guarantee that they will always find cycles.

The first method is to start from 1, 2, 8, or 16 and iterate $g$ on it a number of times before applying the normal Collatz Function to the number generated. Assuming the Collatz Conjecture is true, this will always result in a cycle since all trajectories must yield 1, 2, 8, and 16, which can be seen by observing the original Collatz graph. The second method is to systematically search through every possible tuple combination of every possible length.

Tuples found by the first method are called trivial tuples since all cycles that these tuples represent found must include 1, 2, 8, or 16 and the method used is relatively simple. Applying the same method to other numbers does not always yield the same result. So, tuples found only by the second method are called non-trivial tuples.

Now, we may state the following conjecture:

\newtheorem{LCC}{Conjecture}
\begin{LCC} For all $n \neq 0 \text{ (mod 3)} \in \mathbb{N}_1$, $n$ is a vertex of a cycle.
\end{LCC}

Clearly, multiples of 3 cannot be in cycles by observing that, when applying $g$, all integers $x$ go to $3x+1$, which is another integer.

\subsection{Proof of Bitwise Satisfaction}
Define $R:(a_1,a_2,\dots,a_{n-1},a_n) \to (a_2,a_3\dots,a_n,a_1),\mathbb{N}_0^n \to \mathbb{N}_0^n$ to be a left bit-wise rotation on an $n$-tuple. 

\newtheorem{Satisfied Tuples}{Theorem}
\begin{LCF}
An n-tuple $X = (x_1,\dots,x_n)$ satisfies the LCF if and only if all bitwise rotations on that tuple also satisfy the LCF.
\end{LCF}

Recall that, for $X$ to satisfy the LCF, $F_L(X)$ must be a positive integer that has $g$ applied to it in the cycle. Furthermore, recall that, by Theorem 1, for an $n$-cycle, $y_k$ is defined as the number of applications of $f$ on a number $x_k$ before applying $g$ again to produce a new number $x_{k+1}$. Since the sequence $x_1,\dots,x_n,x_1,\dots$ is periodic, so too is the sequence $y_1,\dots,y_n,y_1,\dots$. We may leverage these two facts to prove this claim.

\begin{proof}
Assume the claim is false. Then, there exists a cycle, in which both integer and non-integer vertices have $g$ applied to them. 

So, there exists a pair of consecutive vertices in the cycle in the Loosened Collatz Graph, such that the first is an integer $x$ and the second is a non-integer. Since, for all $x \in \mathbb{N}_1$, $g(x) \in \mathbb{N}_1$, then the non-integer must be $f(x)$ or $\frac{x}{2}$. 

Let $a,b \in \mathbb{N}_1$ such that $\frac{a}{2^b} \not\in\mathbb{N}_1$. Applying $g$ produces $\frac{3a+2^b}{2^b}$, and since $2^b \nmid a$, $2^b \nmid 3a+2^b$ either. Thus, $\frac{3a+2^b}{2^b}$ is a non-integer of the form $\frac{c}{2^d}$, where $c,d \in \mathbb{N}_1$ such that $2^d \nmid c$.
Likewise, applying $f$ produces $\frac{a}{2^{b+1}}$, another non-integer of the form $\frac{c}{2^d}$, where $c,d \in \mathbb{N}_1$ such that $2^d \nmid c$.
Thus, all non-integer vertices in the cycle must be of the form $\frac{m}{2^n}$, for some $m,n \in \mathbb{N}_1$ such that $2^n \nmid m$.

However, since applying $g$ or $f$ on any non-integer of the form expressed above results in another non-integer of that form, it follows that there cannot exist integer vertices in the cycle after the initial non-integer. But this contradicts our initial assumption that there exists a cycle with both integer and non-integer vertices.

Thus, if $X$ satisfies the LCF, then so too do all bitwise rotations of it. Likewise, if all bitwise rotations of $X$ satisfy the LCF, then so too must $X$, since all outputs must all be either positive integers or positive non-integers.

\end{proof}

Theorem 2 shows that, for any $n$-tuple that we test with the LCF, we are effectively testing at most $n$ different tuples, which are the bitwise rotations of the $n$-tuple. So, Theorem 2 provides an easy way for processing which tuples have been tested, helping us with the later search for non-trivial tuples.

From this, we may state the following conjecture:

\newtheorem{Conjecture1}{Conjecture}
\begin{LCC}
Excluding the tuple $(2,2,...)$, a tuple satisfying the Loosened Collatz Function must include $0$ in its entry.
\end{LCC}

If this conjecture is true, then there exist no cycles in the original Collatz graph aside from the trivial cycle $\dots\xrightarrow{f}1 \xrightarrow{g} 4 \xrightarrow{f} 2 \xrightarrow{f}\dots$.

\subsection{Proof of Unique Cycle Representation}
For later sections, the following theorem and its argument will prove useful:

\newtheorem{Unique Cycles}{Theorem}
\begin{LCF}
If $A = (a_1,\dots,a_n)$ and $B = (b_1,\dots,b_n)$ are $n$-tuples that satisfy the LCF such that $F_L(R^k(A))=F_L(R^k(B))$ for all $k \in \mathbb{N}_1$, then $A=B$.
\end{LCF}

\begin{proof}
Recall that, in proving Theorem 1, we constructed a family of equations which represented a cycle and, from there, constructed a tuple, which has the number of applications of $f$ before applying $g$ as entries. If $F_L(R^k(A))=F_L(R^k(B))$ for all $k \in \mathbb{N}$, then the families of equations the tuples represent are:

$x_2 = f^{a_1}(g(x_1)) = f^{b_1}(g(x_1))$

$x_3 = f^{a_2}(g(x_2)) = f^{b_2}(g(x_2))$

$\vdots$

$x_n = f^{a_{n-1}}(g(x_{n-1}) = f^{b_{n-1}}(g(x_{n-1}))$

$x_1 = f^{a_n}(g(x_n)) = f^{b_{n}}(g(x_{n}))$\\
Thus:

$\frac{3x_1 + 1}{2^{a_{1}}} = \frac{3x_1 + 1}{2^{b_{1}}}$

$\frac{3x_{2} + 1}{2^{a_{2}}} = \frac{3x_{2} + 1}{2^{b_{2}}}$

$\vdots$

$\frac{3x_{n-1} + 1}{2^{a_{n-1}}} = \frac{3x_{n-1} + 1}{2^{b_{n-1}}}$

$\frac{3x_n + 1}{2^{a_{n}}} = \frac{3x_n + 1}{2^{b_{n}}}$\\
It is easy to see that $a_k = b_k$ for all $1 \leq k \leq n$.

\end{proof}

Theorem 3 shows that, for every cycle in the Loosened Collatz Graph, it can be represented by a unique tuple, along with its bitwise rotations. It may be possible to discover a normal form for such satisfying tuples, though we won't do so here.

\section{The Monoid Structure of Satisfying Tuples}
It may be useful to explore directed circuits more generally, instead of cycles as we have defined them. This can be done by investigating tuples instead of cycles in the Loosened Collatz Graph, which leads us to asking whether there exist any algebraic properties of tuples. Further, we ask whether it is possible to construct further tuples that satisfy the LCF. With guidance from an article introducing monoids (\citeauthor{monoids}), we investigate how cycles may be interpreted as such.

Define $S_k$ to be the union of the set of all tuples $X$ with non-negative integer entries, such that $F_L(X) = k$ for some positive integer $k$, and the empty set. So, $S_k = \{X=(x_1,\dots,x_n), \space x_i \in \mathbb{N}_0 \text{ }|\text{ } F_L(X) = k \} \space \cup \space \{ \emptyset \}$, for some $k \in \mathbb{N}_1$. Denote the binary operation of concatenation between two tuples $X = (x_1,\dots,x_p)$ and $Y = (y_1, \dots, y_q)$ with $\cdot$, such that $X \cdot Y = (x_1,\dots, x_p, y_1, \dots, y_q)$. It will be also useful to denote the number of entries of a tuple $Z = (z_1,\dots,z_r)$ as $[Z] = r $ for $r \in \mathbb{N}_0$.

\newtheorem{Sk}{Lemma}
\begin{Sk}
If $F_L(X)=F_L(Y)=k$ for some $X,Y \in S_k$, then $F_L(X \cdot Y)=F_L(Y \cdot X) = k$. Thus, $S_k$ is closed under concatenation.
\end{Sk}

\begin{proof} 
Recall the family of equations constructed in the proofs for Theorem 1 and Theorem 3. If $F_L(X)=F_L(Y)=k$, where $X=(x_1,\dots,x_p)$ and $Y=(y_1,\dots,y_q)$, then

$a_2 = f^{x_1}(g(k))$ 

$a_3 = f^{x_2}(g(a_2))$

$\vdots$ 

$a_p = f^{x_{p-1}}(g(a_{p-1}))$

$k = f^{x_p}(g(a_p))$\\
is the family of equations for the tuple $X$, where $a_2,\dots,a_p\in \mathbb{N}_1$ and $k=a_1$, and

$b_2 = f^{y_1}(g(k))$ 

$b_3 = f^{y_2}(g(b_2))$

$\vdots$ 

$b_q = f^{y_{q-1}}(g(b_{q-1}))$

$k = f^{y_q}(g(b_q))$\\
is the family of equations for the tuple $Y$, where $b_2,\dots,b_q \in \mathbb{N}_1$ and $k=b_2$. From this, let us observe the following family of equations:

$c_2 = f^{x_1}(g(c_1))$ 

$c_3 = f^{x_2}(g(c_2))$

$\vdots$ 

$c_p = f^{x_{p-1}}(g(c_{p-1}))$

$c_{p+1} = f^{x_p}(g(c_p))$

$c_{p+2} = f^{y_1}(g(c_{p+1}))$ 

$c_{p+3} = f^{y_2}(g(c_{p+2}))$

$\vdots$ 

$c_{p+q} = f^{y_{q-1}}(g(c_{p+q-1}))$

$c_{p+q+1} = f^{y_q}(g(c_{p+q}))$\\
where $c_1 = k$. With reference to the previous two families, we see that $c_i = a_i$, for $ 1 \leq i \leq p$, and so $c_{p+1} = k$. Likewise, $c_{p+j} = b_j$, for $1 \leq j \leq q$, and so $c_{p+q+1} = k$. Therefore, we observe that $F_L((x_1,\dots,x_p,y_1,\dots,y_q)) = k$. Since $(x_1,\dots,x_p,y_1,\dots,y_q) = (x_1,\dots,x_p) \cdot (y_1,\dots,y_q) = X \cdot Y$, it follows that, if $F_L(X)=F_L(Y)=k$, then $F_L(X \cdot Y) = k$. A similar argument can be used to show that $F_L(Y \cdot X) = k$ also holds.
\end{proof}

Thus, the set $S_k$ is closed under concatenation. It is worth remarking that $S_k$ equipped with concatenation is not commutative under equality, since, for $X,Y \in S_k\setminus \{\emptyset\}$, $X \cdot Y \neq Y \cdot X$, even though the circuits they represent in the Loosened Collatz Graph pass through the same edges and vertices the same number of times. So, we denote $A \simeq B$ for tuples $A$ and $B$ $\in S_k$ if, when observing the circuits of the graph, the unordered sets of edges and vertices in the circuits represented by $A$ and $B$, as well as the number of times each edge and vertex has been met, are equal. We can say that $A$ and $B$ are equivalent, since the binary relation $\simeq$ is an equivalence relation. So, $X \cdot Y \simeq Y \cdot X$, since both circuits have the same set of edges and vertices and each edge and vertex has been passed through the same number of times in both $X \cdot Y$ and $Y \cdot X$. It is also clear to see that $A \simeq B$ requires $A$ and $B$ to have the same number of entries, since the cycles they represent must have $g$ applied an equal number of times in both.

From this, we can derive the following theorem:

\newtheorem{MonoidEquivalence}{Theorem}
\begin{LCF}
The set $S_k$ equipped with concatenation is a commutative, cancellative monoid under equivalence.
\end{LCF}

\begin{proof}
Note that a set with its binary operator is a monoid if it follows the associative law and contains an identity element. Let $X,Y,Z \in S_k$ be tuples $X=(x_1,\dots,x_p),Y=(y_1,\dots,y_q)$ and $Z=(z_1,\dots,z_r)$. Observe that:

\begin{equation}
\begin{split}
    (X \cdot Y) \cdot Z &  \simeq (x_1,\dots,x_p,y_1,\dots,y_q) \cdot (z_1,\dots,z_r) \\
    & \simeq (x_1,\dots,x_p,y_1,\dots,y_q,z_1,\dots,z_r)\\
    & \simeq (x_1,\dots,x_p) \cdot (y_1,\dots,y_q,z_1,\dots,z_r)\\
    & \simeq X \cdot (Y \cdot Z)
\end{split}
\end{equation}
So, the associative law holds. It is clear to see that, for all $X \in S_k$,

$X \cdot \emptyset \simeq \emptyset \cdot X \simeq X$.\\
So, there exists an identity element, which proves that $S_k$ equipped with concatenation is a monoid under equivalence. If $S_k$ is cancellative, then $X \cdot Y \simeq X \cdot Z$ implies that $Y \simeq Z$. Following from above, if $X \cdot Y \simeq X \cdot Z$, then

$(x_1,\dots,x_p,y_1,\dots,y_q) \simeq (x_1,\dots,x_p,z_1,\dots,z_r)$.\\
From this, observe that the $[X \cdot Y]$ and $[X \cdot Z]$ must be equal. So, $p+q = p+r$, which leads to $q=r$. Further, by equivalence, $X \cdot Y$ and $X \cdot Z$ have the same unordered set of edges and vertices encountered in their circuits as well as the number of times each edge and vertex has been met. Using the family of equations argument from Theorem 3, $Y$ and $Z$ must be equivalent to keep the unordered set of edges and vertices equal after removing the circuit represented by $X$. This leads to the implication that $Y \simeq Z$.
\end{proof}

One could also prove that $S_k$ is a non-commutative, cancellative monoid under equality as opposed to equivalence, though this will not be as useful later. Since $S_k$ is a commutative, cancellative monoid under equivalence, we can make a few key definitions and observations about elements in $S_k$, adapted from a paper on Unique Factorisation in Abstract Algebra (\cite{UFM}):

\begin{itemize}
    \item $X | Y$ in $S_k$ if $X \cdot Z \simeq Y$ for some $Z \in S_k$.
    \item Define $U \in S_k$ to be a unit in $S_k$ if there exists an $V \in S_k$ such that $U \cdot V \simeq V \cdot U \simeq \emptyset$, since $\emptyset$ is the identity. Clearly, the only unit is $\emptyset$, since this is the only element that has an 'inverse', itself.
    \item Define $S_k^{*}$ to be the set of elements of $S_k$ that are not units, which is equal to $S_k \setminus \{\emptyset\}$.
    \item If $X \in S_k^{*}$ and $Y,Z \in S_k$, then we say that:
    \begin{itemize}
        \item $X$ is prime if, whenever $X|Y \cdot Z$, $X|Y$ or $X|Z$.
        \item $X$ is irreducible or an atom if, whenever $X \simeq Y \cdot Z$, then $Y$ or $Z$ is a unit in $S_k$.
    \end{itemize}
    \item Define $A(S_k)$ to be the set of tuples that are irreducible in $S_k$. By using the family of equations construction for any tuple in $S_k$, we can see that $A(S_k)$ contains the set of all tuples representing circuits that apply $g$ to $k$ once, which do not necessarily have to be cycles. Thus, for all $X \in A(S_k)$, $F_L(R^l(X)) \neq k$ for $1 \leq l < [X]$
    \item $A,B \in S_k$ are associates if $A \simeq B \cdot U$ for some unit $U$. Since the only unit is $\emptyset$, an associate of any element is simply itself.
\end{itemize}

From here, we can make connection between all primes and all atoms.

\newtheorem{UFM}{Corollary}
\begin{UFM}
All primes are atoms and all atoms are primes in $S_k$.
\end{UFM}

\begin{proof}
Supposing that $P \in S_k$ is prime, if $P \simeq A \cdot B$ for some $A,B \in S_k$, then $P|A$ or $P|B$. If $P|A$, then $A \simeq P \cdot C$ for some $C \in S_k$. Thus, $P \simeq A \cdot B \simeq P \cdot C \cdot B$ and, since $S_k$ is cancellative, $P \simeq P \cdot C \cdot B$ implies $\emptyset \simeq C \cdot B$ and so $C$ and $B$ is a unit. Likewise, we may apply the same logic if $P|B$ to deduce that $A$ is a unit. Since $P$ is neither a unit nor can be expressed as non-units, $P$ is atomic.

Supposing $X \in S_k$ is atomic, if $X \simeq Y \cdot Z$, then either $Y$ or $Z$ is a unit of $S_k$. If $Y$ is a unit, then $Y = \emptyset$ since that is the only unit of $S_k$, which leads to $X \simeq Z$, by the identity law. If $X \simeq Z$, then $X|Z$, so $X$ is also prime. We may apply the same logic if $Z$ is a unit instead of $Y$ to deduce that $X|Y$. This shows that all primes are atomic and all atoms are primes.
\end{proof}
We shall use them interchangeably now.

\newtheorem{unique}{Corollary}
\begin{UFM}
$S_k$ is a unique factorisation monoid (UFM).
\end{UFM}

\begin{proof}
For $S_k$ to be a UFM, the following conditions must hold for any element $X \in S_k^{*}$:

\begin{enumerate}
  \item There exist atoms $\alpha_1,\dots,\alpha_k \in S_k$ such that $X \simeq \alpha_1 \cdot \dots \cdot \alpha_k$. 
  \item If $\alpha_1,\dots,\alpha_k$ and $\beta_1,\dots,\beta_j$ are atoms of $S_k$ and $X \simeq \alpha_1 \cdot \dots \cdot \alpha_k \simeq \beta_1 \cdot \dots \cdot \beta_j$, then $k=j$ and there is a unique $s$ for each $t$ such that $\alpha_s = \beta_t$.
\end{enumerate}

To prove the first condition, let $X=(x_1,\dots,x_p)$ be an element of $S_k^{*}$, where $[X]=p$ for some $p \in \mathbb{N}_1$. It is either an atom or not: if it is an atom, then $X = \alpha_1$ for some atom $\alpha_1$; if it is not an atom, then it can be decomposed further into atoms. If so, let $X= \alpha_1 \cdot A_1$, where $\alpha_1$ is an atom and $A_1$ is an element of $S_k^{*}$. Since atoms cannot be the empty set, $[A_1] < [X] = p$. 

If $A_1$ is an atom, we may stop there. If $A_1$ is not an atom, then it can be decomposed further into atoms. If so, let $A_1 = \alpha_2 \cdot A_2$, where $\alpha_2$ is an atom and $A_2$ is an element of $S_k^{*}$. Since atoms cannot be the empty set, $[A_2] < [A_1] < [X]$. 

If we assume that there exists an $X$ that is neither an atom nor can be decomposed into atoms, then this process can occur infinitely many times. However, by the well-ordering principle, it cannot, since eventually there will exist an $A_k$ for some $k \in \mathbb{N}$ such that $[A_k] \leq 0$, which is not possible. Thus, there are finitely many times one can apply this process onto some $X \in S_k^{*}$. So, there will always exist atoms $\alpha_1,\dots,\alpha_k \in S_k$ such that $X \simeq \alpha_1 \cdot \dots \cdot \alpha_k$.

To prove the second condition, suppose that $X \simeq \alpha_1 \cdot \dots \cdot \alpha_k \simeq \beta_1 \cdot \dots \cdot \beta_j$, where $\alpha_1, \dots, \alpha_k$ and $\beta_1, \dots, \beta_j$ are primes. Assuming that $k \geq j$, we may consider $\alpha_1$ and notice that, by definition of being a prime, $\alpha_1|\beta_t$ for some $1 \leq t \leq j$. Since $\alpha_1$ and $\beta_t$ are both primes, they are associates, but since all associates are simply themselves, $\alpha_1 = \beta_t$. 

We may relabel $\beta_t$ as $\beta_1$ and $\beta_1$ to $\beta_t$, unless $t = 1$. So, $\alpha_1 \cdot \dots \cdot \alpha_k \simeq \beta_1 \cdot \dots \cdot \beta_j \simeq \alpha_1 \cdot \beta_2 \cdot \dots \cdot \beta_j$ and by the cancellative property, we obtain $\alpha_2 \cdot \dots \cdot \alpha_k \simeq \beta_2 \cdot \dots \cdot \beta_j$. We may repeat this process, cancelling out each $\alpha_s$ with a $\beta_t$. Supposing that not all atoms of $\alpha$ have been cancelled out by the end of this process, since $k \geq j$, we obtain $\alpha_{j+1} \cdot \dots \cdot \alpha_k \simeq \emptyset$. 

But this implies that the tuples $\alpha_{j+1}, \dots, \alpha_k = \emptyset$, so they are not truly atoms. This means that that the only tuples which are actually primes are $\alpha_1, \dots, \alpha_j$ and $\beta_1, \dots, \beta_j$, showing that the number of primes of both factorisations must be equal and that there is a unique $s$ for each $t$ such that $\alpha_s = \beta_t$. Using the same logic when assuming $j \geq k$ leads to us proving the second condition fully.
\end{proof}

This result implies that satisfying tuples can be decomposed into atomic/prime satisfying tuples. This is confirmed in the following confirmation:

\newtheorem{algorithm}{Observation}
\begin{algorithm}
There exists a prime decomposition for satisfying tuples, unique up to equivalence.
\end{algorithm}

\begin{proof}
Recall that, if $F_L(X) = F_L(Y) = k$ for some $X,Y \in S_k$, then $F_L(X \cdot Y) = F_L(Y \cdot X) = k$ and so $X \cdot Y$ and $Y \cdot X$ are also in $S_k$. Usually, the inverse is not true: $X \cdot Y \in S_k$ does not imply $X,Y \in S_k$ for any tuple $X,Y$, unless $F_L(X \cdot Y) = F_L(Y \cdot X) = k$. To see why, observe that the family of equations produced for any tuple $D = (d_1,\dots,d_n)$ in $S_k^{*}$:

$c_2 = f^{d_1}(g(c_1))$ 

$c_3 = f^{d_2}(g(c_2))$

$\vdots$ 

$c_p = f^{d_{p-1}}(g(c_{p-1}))$

$c_{p+1} = f^{d_p}(g(c_p))$

$c_{p+2} = f^{d_{p+1}}(g(c_{p+1}))$ 

$c_{p+3} = f^{d_{p+2}}(g(c_{p+2}))$

$\vdots$ 

$c_{n} = f^{d_{n-1}}(g(c_{n-1}))$

$c_{1} = f^{d_n}(g(c_{n}))$\\
where $c_1 = k$. If there exists a $1 < p < n$, such that $c_{p+1} = k$, then the tuples $(d_1,\dots,d_p)$ and $(d_{p+1},\dots,d_n)$ are also in $S_k$. Let $X= (d_1,\dots,d_p)$ and $Y=(d_{p+1},\dots,d_n)$. Notice that $D = X \cdot Y = R^{[Y]}(Y \cdot X)$. 

Thus, to find atoms in a tuple $D$, we may bitwise rotate that tuple, until $F_L(R^{l}(D)) = k$ for $1 \leq l < [D]$, where $l$ is the smallest such value, if that occurs. If it does not occur, then by definition $D$ is an atom. If it does occur, then we can note that $D$ decomposes into the tuples $(d_1,\dots,d_l)$ and $(d_{l+1},\dots,d_{[D]})$. Since $l$ is the smallest value such that $F_L(R^{l}(D)) = k$, then the tuple $(d_1,\dots,d_l)$ is an atom by definition. We may continue this process on the tuple $(d_{l+1},\dots,d_{[D]})$, decomposing it further into atoms. Thus, we have an algorithm for decomposing a tuple into a product of atoms over concatenation.
\end{proof}

Thus, to investigate properties of satisfying tuples, it is sufficient to consider only those that are prime. This reduces the number of satisfying tuples of interest substantially.

\section{N-Dimensional Objects, representing Cycles}
Suppose that we interpreted the $n$-tuple $Y=(y_1,\dots,y_n)$ as a coordinate in $\mathbb{R}^n$. By Theorem 2, we know that, if one coordinate satisfies LCF, then so do all bit wise rotations of it. It may be worth looking at the object generated by these coordinates. One way to do so is to connect consecutive points together, creating an $n$-dimensional polygon. Let $S=(Y,R(Y),\dots,R^n(Y))$ be the set of vertices in $\mathbb{R}^n$ and $L=(L_1,L_2,\dots,L_n)$ be the set of edges in $\mathbb{R}^n$, such that $L_k$ is the edge from $R^{k-1}(Y)$ to $R^{k}(Y)$, for $1\leq k < n$, and $L_n$ is the edge from $R^{n-1}(Y)$ to $Y$. We denote $|L_k|$ to be the length of edge $L_k$. We then observe the trivial observation:

\newtheorem{Lemma}{Observation}
\begin{Lemma}
For all $k \in \mathbb{N}_1$, where $1 \leq k \leq n$, $|L_k|$ = $|L_{k+1}|$.
\end{Lemma}

This holds for any n-tuple, regardless of whether it satisfies the LCF or not. Recall that a cycle is a simple directed circuit in the Loosened Collatz Graph, and that, in this cycle, $g$ is applied $n$ many times. This is in fact the length of the associated satisfying tuple, $S$. We can thus denote the number of times $g$ is applied in a cycle $C$ as $|S| = n $. Also note that, by Theorem 2, for every vertex that has $g$ applied to it in $C$, it can be represented as bitwise rotations of a single $n$-tuple. Interpreting these tuples as coordinates in $\mathbb{R}^n$ and the n-dimensional object created as above, we can provide the following theorem.

\newtheorem{DualShape}{Theorem}
\begin{LCF}
For each and every cycle $C$ in the Collatz Function, there exists a unique object in $\mathbb{R}^{|S|}$.
\end{LCF}

\begin{proof}
Theorem 2, combined with the interpretation of tuples as coordinates, shows that every cycle has a unique object. Theorem 3 proves that every cycle can be represented by a unique tuple, along with its bitwise rotations. Together, this proves the above theorem.
\end{proof}

Thus, we have found a neat one-to-one correspondence between solutions to the LCF and their geometric counterparts. 
From this, an equivalent conjecture to Conjecture 2 states:

\newtheorem{Conjecture4}{Conjecture}
\begin{LCC}
Excluding the point $(2,2,...)$, all dual objects of satisfying cycles require their vertices to be 0 on at least one axis.
\end{LCC}

We also conjecture that this object is invariant under some rotation about the line $x_1=x_2= \dots = x_n$, where $x_i$ corresponds to the $i$th axis.

If either Conjecture 2 or Conjecture 4 are true, then there exist no cycles in the original Collatz graph aside from the trivial 1,2,4 cycle. Of course, there is currently no use for this observation, but perhaps those with greater knowledge and a more varied skill set may use Theorem 5 to investigate the conjectures. 

\section{Conclusion}

We have investigated the Loosened Collatz Graph, derived the LCF and shown how they relate to tuples satisfying the LCF, which represent circuits. We have shown the relationship between a tuple and its bitwise rotations and the satisfaction of a tuple, given its bitwise rotations, and vice versa, suggesting a possible normal form for such satisfying tuples. We have proven that the set of tuples which produce the positive integer $k$ in the LCF form a unique factorisation monoid, which may help aid future researchers into proving properties of the atoms in the monoid and the structure of tuples in the set. We have suggested an interpretation of tuples as coordinates, which may provide geometric insights into the problem, though we currently lack the knowledge to explore this path. Future researchers may be interested in proving or disproving the three main conjectures of the paper:
\begin{enumerate}
    \item For all $n \neq 0 \text{ (mod 3)} \in \mathbb{N}_1$, $n$ is a vertex of a cycle.
    \item Excluding the tuple $(2,2,...)$, a tuple satisfying the Loosened Collatz Function must include $0$ in its entry.
\end{enumerate}
\section*{Acknowledgements}

We owe thanks to Maiesha Siddika, Owen Mackenzie, Akira Wan, Emils Bahanovskis, Daniel Espinoza, Tervel Valchanov and others at King's College London School of Mathematics for useful discussion at the beginning of research. We also thank Layo Danbury, Sophia Gregorio, Na Wang, Saanya Verma and others for their support during our research. We are grateful for the help Edward Smith's father provided us with uploading the necessary files to GitHub and collecting the results for Section 3.2.
We are grateful for the services provided by King's College London School of Mathematics, as without, we would not have been able to perform our research.

\bibliography{main}

\end{document}